\documentclass[12pt]{amsart}
\usepackage{amsfonts}
\usepackage{amsmath}
\usepackage{amsxtra}
\usepackage{amssymb,latexsym}
\usepackage[mathcal]{eucal}
\usepackage{amscd}
\usepackage{tikz-cd}
\usepackage{hyperref}

\newtheorem{theo}{{\bfseries Theorem}}[section]

\newtheorem{lem}[theo]{{\bfseries Lemma}}
\newtheorem{cor}[theo]{{\bfseries Corollary}}
\newtheorem{df}[theo]{{\bfseries Definition}}

\newtheorem{ex}[theo]{{\bfseries Example}}

\def \ol {\overline}
\def \N {\mathbb N}

\def \Z {\mathbb Z}

\def \A {\mathcal A}
\def \B {\mathcal B}
\def \CC {\mathcal C}

\def \NN {\mathcal N}

\def \O {\mathcal O}

\def \a {\alpha }

\def \b {\beta}

\def \ep {\epsilon}
\def \om {\omega}
\def \th {\theta}
\def \d {\delta}

\def \1  {{\mathbf 1}}
\def \0  {{\mathbf 0}}

\usepackage{amssymb,latexsym}
\usepackage[mathcal]{eucal}
\usepackage{amscd}
\usepackage{amsmath}
\usepackage{graphicx}
\usepackage{amsfonts}
\usepackage{amscd}

\numberwithin{equation}{section}
\begin{document}

\begin{titlepage}
\large
\title {\bfseries Realizing Arbitrary Depth}
\author{Ethan Akin}
 \vspace{.7cm}

\address{Mathematics Department \\
    The City College \\ 137 Street and Convent Avenue \\
       New York City, NY 10031, USA     }
\email{ethanakin@earthlink.net}

\date{July, 2026} \vspace{.7cm}

{\footnotesize \begin{abstract}We provide a simple construction which realizes the Birkhoff center depth at an arbitrary ordinal level and relate it
to the Cantor-Bendixson depth.
\end{abstract}}



\end{titlepage}
\maketitle

For Benjy Weiss: An inspiration and a joy.\vspace{.7cm}

\setcounter{page}{1}
\section{ \textbf{Introduction}}\label{intro}\vspace{.5cm}

For us, a dynamical system is a pair $(X,f)$ with $X$ a compact metric space and $f : X \to X$ a homeomorphism.

We are concerned here with three notions of topological
depth.

In each case there is defined a transfinite decreasing sequence of closed subsets, indexed by the ordinals. At each step the space $X$ is replaced by
a closed subspace $z(X)$. In detail, we
begin with $X_0 = X$ and then proceed by defining
\begin{equation}\label{eq0.01}\begin{split}
X_{\a + 1} \ = \ z(X_{\a}), \hspace{2cm} \\
X_{\a} \ = \ \bigcap_{\b < \a} \ X_{\b} \quad \text{for} \ \a \ \text{a limit ordinal}.
\end{split}\end{equation}
The sequence stabilizes when $X_{\a} = X_{\a + 1}$ in which case, $X_{\a} = X_{\b}$ for all $\b$ with  $\a < \b$. We denote by $\th(X)$ the
minimum such $\a$ and refer to it as the associated \emph{depth}. Since we are restricting attention to compact metric spaces,
the depth is a countable ordinal.

For the, purely topological, Cantor-Bendixson procedure, $z_{CB}(X)$ is the derived subset of accumulation points of $X$. That is,
$z_{CB}(X)$ is the complement of the open set of isolated points of $X$. A nonempty compact metric space with no isolated points is
said to be a \emph{perfect} space.
Any nonempty open subset of a perfect space contains a Cantor set through each point and so is uncountable. So if $z_{CB}(X)= X$, then $X$ is perfect or
empty. On the other hand, a compact metric space admits
only countably many isolated points. So if $\th = \th_{CB}(X)$, then $X \setminus X_{\th}$ is countable since $\th$ is a countable ordinal. Thus,
$X_{\th}$ is perfect if and only if $X$ is uncountable.  If $X$ is countable, then $X_{\th} = \emptyset$. Because the monotone intersection of
nonempty compacta is nonempty, it follows that when $X$ is countable, $\th_{CB}(X)$ cannot be a limit ordinal. The Cantor-Bendixson
center  $X_{\th}$ is $\{ x \in X : U$ is uncountable for every open set $U$ such that $x \in U \}$. It is the union of the
Cantor subsets of $X$.

For a dynamical system $(X,f)$, there are two versions of
the Birkhoff procedure, $z_{BN}(X)$ is the set of non-wandering points of $X$ and
$z_{BL}(X)$ is the closure of the union of the limit point sets of points
of $X$. Such limit points are always non-wandering and so $z_{BL}(X) \subset z_{BN}(X)$. However, $X = z_{BL}(X)$ if and only if $X = z_{BN}(X)$ because
each is equivalent to the condition that the recurrent points are dense. Thus, both procedures converge to the same subset $X_{\th}$, the
\emph{Birkhoff center}, \cite{B}. However, the descent via the limit points might be much faster than that via the non-wandering points.  It is
clear that $\th_{BL} \le \th_{BN}$ but the inequality might be strict.

Inspired by the flow examples of Maier \cite{M}(see also flow examples in Neumann \cite{N}), Shapovalov \cite{S00} constructed for every countable ordinal $\a$ a system $(X^{\a},f^{\a})$ such that $\th_{BN}(X^{\a}) = \a$, thus realizing every depth. His examples are countable, invertible subshifts with Birkhoff center consisting of a single
fixed point. Moreover, the fixed point is the alpha limit and the omega limit of every point. Thus, $\th_{BL}(X^{\a}) = 2$.

Other versions of the realization theorems occur in the literature. Theorem 5.35 of \cite{AG19} provides examples which are countable, weakly almost periodic subshifts. Kato in \cite{K} has examples on dendrites and Euclidean balls as well as countable examples. Alsed\`{a} et al \cite{ACS}
obtain the point version of depth results for maps of the interval.

I would like to thank the referee for pointing out some valuable references that I had missed.

The examples we construct are simple in the following sense.\vspace{.25cm}

\begin{df}\label{def0.01} A dynamical system $(X,f)$ is a \emph{simple system} when it satisfies the following conditions.
\begin{itemize}
\item The space $X$ is countable.
\item There is a fixed point $e \in X$ which is the unique recurrent point.
\item Every non-isolated point is an omega limit point of some isolated point.
\end{itemize}

The system is a \emph{trivial system} when $X$ is the singleton $\{ e \}$. \end{df} \vspace{.25cm}

An isolated point is non-wandering if and only
if it is periodic. So if $X$ contains no periodic isolated points, then $z_{BL}(X) \subset z_{BN}(X) \subset z_{CB}(X)$. For a non-trivial
simple system
\begin{equation}\label{eq0.02}
z_{BL}(X) = z_{BN}(X) = z_{CB}(X).
\end{equation}

Our examples will be simple systems at each level and so the three transfinite
sequences will be the same until the Birkhoff center $\{ e \}$ is reached.  Thus, with $\th = \th_{BN} = \th_{BL}$ we have $X_{\a}$
the same for all three procedures for $\a \le \th$  and $X_{\th} = \{ e \}$. The fixed point is recurrent and so the Birkhoff sequences stop there.
Because $e$ is an isolated point, the Cantor-Bendixson sequence has one more step, so that $\th_{CB} = \th + 1$ with $X_{\th + 1} = \emptyset$.

Our main result is the construction, for every countable ordinal $\th$, of a simple example with $\th_{BN} = \th_{BL} = \th$.\vspace{.25cm}

\begin{theo}\label{theo0.02} If $\th$ is a countable ordinal, then there exists a dynamical system $(X,f)$ such that
for every ordinal $\a$ with $0 \le \a < \th$ the system $(X_{\a},f|X_{\a})$ is a non-trivial simple system and  with $X_{\th} = \{ e \}$. Thus,
$\th = \th_{BN}(X) = \th_{BL}(X)$ and $\th + 1 =  \th_{CB}(X).$\end{theo}\vspace{.5cm}

\section{ \textbf{Transitivity and Recurrence}}\vspace{.5cm}

 As usual we use $\N$ for the
set of positive integers and $\Z$ for the entire set of integers. We let  $\Z^*$ denote the
one point compactification of $\Z$ with the point at infinity labelled $*$.  On $\Z^*$ we use the metric (actually an ultra-metric) given by:
\begin{equation}\label{eq1.00}
d(m,n) = \max(1/(|n|+1),1/(|m|+1)),\ \ \text{for} \ \ m \not= n,
\end{equation}
with $1/(* + 1)$ defined to be $0$.

Recall that for a sequence $\{ A_n \}$ of closed subsets of $X$, the closure of the union consists of the union together with the limsup of the sequence, i.e.
\begin{equation}\label{eq1.03}\begin{split}
\ol{\bigcup_n \ A_n} \ = \ (\bigcup_n \ A_n) \cup \limsup_n \ A_n, \\
\text{where} \quad \limsup_n \ A_n \ = \ \bigcap_N \ \ol{\bigcup_{n \ge N} \ A_n}.
\end{split}\end{equation}
If the $A_n$'s are nonempty, then, by compactness, the limsup is nonempty.

For a dynamical system $(X,f)$ the orbit of a point $x \in X$ is given by
 $\O f(x) = \{ f(x), f^2(x), \dots \}$ and $\om f(x)$ is the set of limit points of the orbit sequence, nonempty by compactness. That is,
\begin{equation}\label{eq1.05}
 \om f(x) \ = \ \bigcap_N \ \ol{\{ f^n(x) : n \ge N \}}.\hspace{2cm}
\end{equation}

We  say that $x$ wanders to $y$, written $y \in \NN f(x)$,
 when for any pair of open sets $U, V$ with $x \in U$ and $y \in V$, there exists $z \in U$ and $k \in \N$
such that $f^k(z) \in V$.

Finally, let $d$ be a metric for $X$. We define the \emph{Conley chain relation} $\CC f$.
For $n \in \N$ and $\ep > 0$ a sequence $\{ x_0, x_1, \dots, x_n \}$ in $X$ is an $\ep$ chain from $x$ to $y$ if
$x = x_0, y = x_n$ and for $k = 1, \dots, n, \ d(f(x_{k-1}), x_k) < \ep$.
\begin{equation}\label{eq1.07}\begin{split}
y \in  \CC f(x) \quad \Longleftrightarrow \hspace{3cm} \\
 \text{For all} \ \ \ep > 0 \ \ \text{there exists an} \ \ep \ \text{chain from} \ x \ \text{to} \ y.
\end{split}\end{equation}
Because the metrics on $X$ are uniformly equivalent, the relation $\CC f$ is independent of the choice of metric.

It is easy to see the inclusions
\begin{equation}\label{eq1.09}
 \O f(x) \cup \om f(x) \ \subset \ \NN f(x) \ \subset \ \CC f(x).
\end{equation}
For the first inclusion, see Theorem \ref{theo1.02} below.

Associated with these relations are various kinds of recurrence for $f$.
\begin{itemize}
\item A point $x$ is a \emph{fixed point} when $f(x) = x$.

\item A point $x$ is a \emph{periodic point} when $x \in \O f(x)$.

\item A point $x$ is a  \emph{recurrent point} when $x \in \om f(x)$.We denote the set of recurrent points
by $Recur_f$.

\item A point $x$ is a   \emph{non-wandering point}  when $x \in \NN f(x)$.

\item A point $x$ is a   \emph{chain recurrent point}  when $x \in \CC f(x)$.

\item A point $x$ is a  \emph{transitive point} for $f$ when $\om f(x) = X$. We denote the,
possibly empty, set of transitive points by $Trans_f$.
\end{itemize}

It is easy to see that  for $\A = \O, \NN$ or $\CC$, we have $y \in \A f(x) \ \Leftrightarrow \ x \in \A (f^{-1})(y)$.
The set $\om (f^{-1})(x)$ is the set of limit points of $\{ f^k(x) \}$ as $k \to - \infty$ and
is labelled $\a f(x)$, and $y \in \om f(x)$ usually does not imply $x \in \a f(y)$. Thus, a point $x$ is periodic, non-wandering or chain recurrent for $f^{-1}$ if and only if it satisfies the corresponding property for $f$, while a point which is recurrent or transitive for $f$ need not be recurrent or transitive for $f^{-1}.$

A set $A \subset X$ is called $f$ \emph{invariant} when $f(A) = A$ and so $A = f^{-1}(A)$.

\begin{df}\label{def1.01} Let $(X,f)$ be a dynamical system.
\begin{itemize}
\item[(a)]The system is called \emph{minimal} when
 every orbit $\O f(x)$ is dense in $X$, or, equivalently,
$Trans_f = X$. The system is minimal if and only if $X$ itself is the only non-empty, closed $f$ invariant subset of $X$.

\item[(b)] The system is called \emph{topologically transitive} when every pair $x, y \in X$ satisfies $y \in \NN f(x)$ and so
whenever   $U, V$ are nonempty open subsets of $X$, there exists $k \in \N$ such that $V \cap f^{-k}(U) \not= \emptyset$
 in which case, the set of such $k$'s is infinite. The system is topologically transitive if and only if
 $Trans_f$ is nonempty, in which  case it is a dense $G_{\d}$  subset of $X$.

 \item[(c)] The system is called \emph{central} when every point $x \in X$ is non-wandering and so
 whenever $U$ is a nonempty open subset of $X$, there exists $k \in \N$ such that $U \cap f^{-k}(U) \not= \emptyset$
 in which case, the set of such $k$'s is infinite. The system is central if and only if the  $G_{\d}$ subset $Recur_f$ is dense in $X$.

 \item[(d)] The system is called \emph{chain transitive} when every pair $x, y \in X$ satisfies $y \in \CC f(x)$.
 The system is chain transitive provided that
 whenever $x$ and $y$  are points of $X$ and $\ep > 0$, there exists an $\ep$ chain from $x$ to $y$. It suffices that there
 exists a point $e \in X$ such that $e \in \CC f(x)$ and $x \in \CC f(e)$ for all $x \in X$.

 \item[(e)] A closed,invariant subset $A$ is called a minimal/topologically transitive/central/chain transitive subset when the
 subsystem $(A, f|A)$ satisfies the corresponding property.
 \end{itemize} \end{df}\vspace{.25cm}

Note that if $f$ is
  minimal, topologically transitive, central or chain transitive, then $f^{-1}$ satisfies the corresponding property.

 By Zorn's Lemma any non-empty, closed $f$ invariant set contains a minimal subset. Any point of a minimal subset is recurrent and so
 any system contains recurrent points. Let $\B$ be a countable basis of non-empty open sets with $\B_n$ the members of $\B$ of diameter
 less than $1/n$. The dense, $G_{\d}$ results in (b) and (c) follow from:
 \begin{equation}\label{eq1.10}\begin{split}
 Trans_f \ = \ \bigcap_{U \in \B, N \in \N} \ \bigcup_{k \ge N} f^{-k}(U). \hspace{1cm}\\
Recur_f \ = \ \bigcap_{n, N \in \N} \ \bigcup_{k \ge N, U \in \B_n} U \cap f^{-k}(U).
 \end{split}\end{equation}

 The sufficiency result in (d) follows from the fact that $\CC f$ is a transitive relation, i.e $e \in \CC f(x)$ and $y \in \CC f(e)$ implies  $y \in \CC f(x)$.

 There are a number of competing definitions of topological transitivity. The alternatives are sorted out in \cite{AC12} and in \cite{BK14}.
Observe that with the definition above, the extension to the one point compactification  $\Z^*$ of the translation $i \mapsto i+1$
defines a system which is not topologically transitive. Thus, a topologically transitive system consists of either
a single periodic orbit or a homeomorphism on a perfect, and so uncountable, space. If $x$ is a recurrent point, then it is a transitive point
for the restriction of $f$ on $\om f(x)$. In particular, if $X$ is countable, then the only recurrent points are periodic. \vspace{.25cm}

\begin{lem}\label{lem1.01a} If a dynamical system $(X,f)$ has a fixed point $e$ as the unique recurrent point, then
$(X,f)$ is chain transitive. \end{lem}

\begin{proof} Since $\{ e \}$ is the unique minimal subset, $e \in \om f(x) \cap \a f(y)$ for all $x, y \in X$. So
(\ref{eq1.09}) implies that $e \in \CC f(x) \cap \CC (f^{-1})(y)$. Hence, $y \in \CC f(e)$ and the result follows from (d) above.
\end{proof}\vspace{.25cm}

If $x$ is a recurrent point, then it is a transitive point for the restriction $(\om f(x), f|\om f(x))$ and so in that case $\om f(x)$ is
a topologically transitive subset. As noted in \cite{A93} Proposition 4.14, it is, in any case, a chain transitive subset. We recall the simple proof.

\begin{theo}\label{theo1.02} Let $(X,f)$ be a dynamical system with $x \in X$.\vspace{.25cm}

The omega limit set $\om f(x)$ is an $f$ invariant, chain transitive subset with $y_2 \in  \NN f(y_1)$ for all
$y_1, y_2 \in \om f(x)$.  In particular, every point of $\om f(x)$ is non-wandering. \end{theo}

\begin{proof} A point $y$ is in $\om f(x)$ if and only if there exists a sequence $\{ k_i \}$ in $\N$ tending to infinity with $\{ f^{k_i}(x) \}$
converging to $y$.  In that case, $\{ f^{k_i+1}(x) \}$ converges to $f(y)$ and  $\{ f^{k_i-1}(x) \}$ converges to $f^{-1}(y)$.
Thus, $\om f(x)$ is $f$ invariant.

Given $\ep > 0$, let $\d < \ep/2$ be an $\ep/2$ modulus of uniform continuity for $f$. From the definition of the limsup and compactness, there exists
$N \in\N$ such that if $k \ge N$, then there exists a point $z_k \in \om f(x)$ with $d(f^k(x),z_k) < \d$. Given $y_1, y_2 \in \om f(x)$
there exist $N \le k_1 < k_2$ such that $d(f^{k_1}(x),y_1) < \d$ and $d(f^{k_2}(x),y_2) < \d$. So we can choose
$z_{k_1} = y_1$ and $z_{k_2} = y_2$. Let $n = k_2 - k_1$. With $u = f^{k_1}(x)$, we have $d(u,y_1), d(f^n(u),y_2) < \ep$. As $\ep$ is arbitrary,
it follows that $(y_1,y_2) \in \NN f$. Furthermore, $\{ y_1 = v_0 = z_{k_1}, v_1 = z_{k_1+1}, \dots, v_n = z_{k_2} = y_2 \} $ is an $\ep$ chain in
$\om f(x)$ from $y_1$ to $y_2$ because $d(z_{k_1+i},f^{k_1+i}(x)) < \d$ and so $d(f(z_{k_1+i}),f^{k_1+i+1}(x)) < \ep/2$ and again
$d(z_{k_1+i+1},f^{k_1+i+1}(x)) < \d < \ep/2$. Thus, $y_2 \in \CC (f|\om f(x))(y_1)$.

\end{proof}\vspace{.25cm}

\section{\textbf{Three Constructions}}\vspace{.5cm}

We first describe an attachment construction  due, I believe,  to Takens.

\begin{df}\label{def1.03} Let $(X,f)$ be a dynamical system. A bi-infinite sequence $\{ x_k : k \in \Z \}$  in $X$ is called a \emph{full, asymptotic
chain} when
\begin{itemize}
\item For every $N \in \N$ the tail sets $\{ x_k : k \ge N \}$ and $\{ x_k : k \le -N \}$ are dense in $X$.

\item The distances $d(f(x_{k-1}),x_k) \to 0$ as $k \to \infty$ and as $k \to -\infty$.
\end{itemize}\end{df}\vspace{.25cm}

\begin{lem}\label{lem1.04} A dynamical system $(X,f)$ is chain transitive if and only if it admits a full, asymptotic chain. \end{lem}

\begin{proof} If $\{ x_k : k \in \Z \}$ is a full asymptotic chain, $y_1, y_2 \in X$, and $\d < \ep/2$ an $\ep/2$ modulus of uniform
continuity for $f$, then we can choose $N$ so that  $d(f(x_{k-1}),x_k) < \d$ for all $k > N$. By the density condition, there exist $k_1, k_2$ with
$N \le k_1  < k_2 - 2$ and such that $d(x_{k_1},y_1) < \d$ and $d(x_{k_2},y_2) < \d$. Then $\{ y_1, x_{k_1+1}, \dots x_{k_2 - 1}, y_2 \}$
is an $\ep$ chain from $y_1$ to $y_2$.

A compact metric space is totally bounded and so there exists for every $\ep > 0$ a finite sequence in $X$ which is $\ep$ dense,
i.e. it passes within $\ep$ of every point of $X$.
By concatenating such finite sequences it is easy to obtain a bi-infinite sequence $\{ y_k : k \in \Z \}$ with dense tails. If $(X,f)$ is chain
transitive, then we can insert between $y_{k}$ and $y_{k+1}$ a $1/(|k| + 1)$ chain from $y_{k}$ and $y_{k+1}$. After renumbering we obtain
a full asymptotic chain.

\end{proof}\vspace{1cm}

For a chain transitive dynamical system $(X,f)$ we fix a full asymptotic chain $\{ x_k : k \in \Z \}$ and
define $aX$ to be the following closed subset of $X \times \Z^*$.
 \begin{equation}\label{eq1.11}
 aX \ = \ (X \times \{ * \}) \cup \{ (x_k,k) : k \in \Z \}.
 \end{equation}
 The map $af$ on $aX$ is defined by
 \begin{equation}\label{eq1.12}
 af(x,*) \ = \ (f(x),*), \qquad af(x_k,k) = (x_{k+1},k+1),
 \end{equation}
 for all  $ x \in X, k \in \Z$.

Clearly, the restriction of $af$ to the subset $X \times \{ * \}$ defines  a subsystem of $(aX, af)$ which is isomorphic to $(X,f)$.
 \vspace{.25cm}

 \begin{theo}\label{theo1.05} Let $(X,f)$ be a chain transitive dynamical system with a full asymptotic chain $\{ x_k : k \in \Z \}$.
 The \emph{attachment pair }$(aX, af)$ is a dynamical system which satisfies the following conditions.

 \begin{itemize}
 \item[(i)] If $X$ is countable, then $aX$ is countable.

 \item[(ii)]  For any $k \in \Z$ the limit sets $\om (af)(x_k,k)$ and $\a(af)(x_k,k)$ equal $ X \times \{ * \}$.

 \item[(iii)]  The system $(aX, af)$ is chain transitive.

 \item[(iv)]  The set $\{ (x_k,k) : k \in \Z \} $ is the set of isolated points of $aX$.

 \item[(v)]  The set $ X \times \{ * \} $ is the set of non-wandering points for $af$.
 \end{itemize}\end{theo}

 \begin{proof} The space $aX$ is compact and metrizable because it is a closed subset of the compact metrizable space $X \times \Z^*$.
 Continuity of $af$ is clear at the points of $\{ (x_k,k) \}$ as these are
 isolated. Since the restriction of $af$ on $X \times \{ * \}$ is continuous it suffices to observe that, using the max metric on the product,
  \begin{equation}\label{eq1.13}\begin{split}
d((x,*),(x_k,k)) = \max(d(x,x_k), 1/(|k|+1)) \hspace{2cm}\\
d((f(x),*),(x_{k+1},k+1)) = \max(d(f(x),x_{k+1}),1/(|k+1|+1))
\end{split}\end{equation}

Given $\ep > 0$, let $\d $ be an $\ep/2$ modulus of uniform continuity for $f$. Let $N \in \N$ be large enough that $1/(N+1) < \ep$ and
$d(f(x_k),x_{k+1}) < \ep/2$ when $|k| \ge N$. If $d((x,*),(x_k,k)) < \min(\d, 1/(N+1))$, then $d(x,x_k) < \d$ and $|k| > N$.
Since $d(f(x),f(x_k)) < \ep/2$, it follows that $d((f(x),*),(x_{k+1},k+1)) \le \ep$. So continuity at $(x,*)$ holds.

It is clear that $af$ is a bijection and so is a homeomorphism.

 The countability claim in (i) is obvious.

The limit set results of (ii) follow because of the density of the tails in the sequence $\{ x_k \}$.

Fix $y \in X$. Since $(y,*) \in \om (af)(x_k,k) \cap \a (af)(x_k,k)$ it follows as in the proof of Lemma \ref{lem1.01a} that
$(y,*) \in \CC (af)(x_k,k)$ and $(x_k,k) \in \CC (af)(y,*)$. If $(X,f)$ is chain transitive, then for all $x \in X$,
$(y,*) \in \CC (af)(x,*)$ and $(x,*) \in \CC (af)(y,*)$.
So (iii) follows from  Definition \ref{def1.01}   (d).

The points $(x_k,k)$ are clearly isolated and the remaining points, those of $X \times \{ * \}$, are in an omega limit set and so
are not isolated in $aX$. Hence, (iv).

From Proposition \ref{theo1.02} and (ii) it follows that the points of $X \times \{ * \}$ are non-wandering. An isolated point is non-wandering only
when it is periodic and so the points $(x_k,k)$ are wandering.

\end{proof}\vspace{.25cm}

We pause to observe the following result. In the thesis of Perez Flores \cite{PF} this is referred to as the Bowen-Sarkovskii Theorem. It is proved in \cite{DF} and in \cite{S65} although in those locations an alternative condition, equivalent to chain transitivity, is used.

 \begin{cor}\label{cor1.06} A dynamical system $(X,f)$ is chain transitive if and only if it is isomorphic to the restriction to the omega limit set
 of a point in some larger system. \end{cor}

  \begin{proof} This is immediate from Theorems \ref{theo1.02}  and  \ref{theo1.05}.

\end{proof}\vspace{.5cm}

Given a dynamical system $(X,f)$ and $N \in \N$ the $N$-fold discrete suspension is $(X_N,f_N)$ with
\begin{equation}\label{eq2.01}\begin{split}
X_N \ = \ X \times \{ 1, \dots, N \}, \quad \text{and with} \ \ f_N \ \ \text{defined by} \\
f_N(x,i) \ = \ \begin{cases} (x,i+1) \ \ \text{for} \ \ i < N, \\ (f(x),1)  \ \ \text{for} \ \ i = N.\end{cases}\hspace{2cm}
\end{split}\end{equation}

Instead, what we will need is the stretched suspension construction.

A \emph{pointed dynamical system} $(X,f)$ is one with base point a specified fixed point $e$.

Let $(X,f)$ be a pointed dynamical system with base point $e$ not isolated.
Assume that $X$ is zero-dimensional, i.e.
the clopen subsets form a basis for $X$. Let $\{ V_1, V_2, \dots \}$ be a decreasing sequence of
clopen subsets of $X$ with $V_1 = X$ and with $\bigcap_k \ V_k \ = \ \{ e \}$. For $k \in \N$ let $D_k = V_k \setminus V_{k+1}$.
Since $e$  is not isolated, infinitely many of these are nonempty and so we can discard the empties and re-number. Thus,
we obtain a partition  $\{ D_k : k \in \N \}$ of $X \setminus \{ e \}$ by nonempty subsets which are clopen in $X$ such that
 $\limsup_k \  D_k \ = \ \{ e \}$. For any $N \in \N$, the clopen set $\bigcup_{k < N} D_k$ does not contain $e$ and
 so its complement in $X, \  V_N = \ol{\bigcup_{k \ge N} D_k },$ is a clopen neighborhood of $e$.

For the pointed, zero-dimensional dynamical system $(X,f)$, base point $e$ and the partition $\{ D_k \}$ we
define $sX$ to be the following subset of $X \times \Z^*$.
 \begin{equation}\label{eq2.02}
 sX \ = \  (\bigcup_{k \in \N}  \  D_k \times [-k,k]) \ \cup \ (\{ e \} \times Z^*),
 \end{equation}
 where $[-k,k]$ is the interval $\{ -k,  -k+1, \dots, k \}$ in $\Z$.

 The map $sf$ on $sX$ is defined so that if $x \in D_k$ and $f(x) \in D_{\ell}$
 \begin{equation}\label{eq2.03}
sf(x,i) \ = \ \begin{cases} (x,i+1) \ \ \text{for} \ \  i < k, \\ (f(x),-\ell)  \ \ \text{for} \ \  i = k.\end{cases}
 \end{equation}
For all $i \in \Z^*$
 \begin{equation}\label{eq2.04}
sf(e,i) \ = \ (e,i+1), \hspace{3cm}
\end{equation}
with $* + 1 = *$.

If $X$ is the trivial pointed system so that $X = \{ e \}$, then by convention we define $sX = s\{ e \} = \{ e \} \times Z^*$  and
with $sf(e, i) = (e, i+1)$ for all $i \in \Z^*$.

Define $T : X \times \Z^* \to X \times Z^*$ by $T(x,i) = (x,-i)$ with $-* = *$ so that $T^2 = id_{sX}$. \vspace{.25cm}

\begin{theo}\label{theo2.01} Let  $(X,f)$ be a pointed, zero-dimensional dynamical system with a non-isolated base point $e$
 and with a partition of $X \setminus \{ e \}$ by a sequence $\{ D_k : k \in \N \}$ of nonempty subsets, clopen in $X$,
  such that $limsup_k \  D_k \ = \ \{ e \}$.
The \emph{stretched suspension} $(sX,sf)$ is a pointed, zero-dimensional dynamical system with base point $(e,*)$
which satisfies the following conditions, where $\pi : sX \to X$ is the
projection to the first coordinate.
 \begin{itemize}
 \item[(i)] The space $sX$ is invariant for the homeomorphism $T$ and the homeomorphism $sf$ on $sX$ satisfies
 $s(f^{-1}) \circ T = T \circ (sf)^{-1}$.

  \item[(ii)]If $X$ is countable, then $sX$ is countable.

 \item[(iii)] If $x \in D_k$ and $i \in [-k,k]$, then
 \begin{equation}\label{eq2.05}\begin{split}
 \pi^{-1}(\om f(x)) \ = \ \om (sf)(x,i), \hspace{2cm}\\
  \pi^{-1}(\a f(x)) \ = \ \a (sf)(x,i), \hspace{2cm}\\
\pi^{-1}(\NN f(x)) \ = \ \NN (sf)(x,i), \hspace{2cm}\\
\pi^{-1}(\CC f(x)) \ = \ \CC (sf)(x,i). \hspace{2cm}
\end{split}\end{equation}
In particular, if $x$ is recurrent, non-wandering or chain recurrent for $f$, then $(x,i)$ satisfies the
corresponding property for $sf$.

\item[(iv)] For all $i \in \Z, \ \ \om (sf)(e,i) \ = \ \a (sf)(e,i) \ = \ \{(e, *) \}$.

\item[(v)] If $x \in D_k$ and $i \in [-k,k]$, then $(x,i)$ is an isolated point of $sX$ if and only if $x$ is an isolated point of $X$.

\item[(vi)] If $(X,f)$ is topologically transitive or chain transitive, then $(sX, sf)$ satisfies the corresponding condition.

\item[(vii)] If $A$ is a closed $f$ invariant subset of $X$ with $e \in A$ and not isolated in $A$, then using
the sequence $\{ D_k \cap A : k \in \N \}$ we obtain the closed $sf$ invariant subset $sA$ with $s(f|A)$ the restriction of $sf$.
If instead $A = \{ e \}$, then $sA$ is by convention a closed $sf$ invariant subset $sA$ with $s(f|A)$ the restriction of $sf$.

\end{itemize}\end{theo}

\begin{proof}  Each member of the sequence $\{ D_k \times [-k,k] \}$ is a clopen subset of $X \times \Z^*$
and the limsup of the sequence is $\{ e \}\times \Z^*$.
Hence, $sX$ is a closed subset of the zero-dimensional, compact set $X \times \Z^*$.

If $x \in D_k$ and $f(x) \in D_{\ell}$ then $D_k \cap f^{-1}(D_{\ell})$ is a clopen neighborhood of $x$ and so
continuity of $sf$ on each clopen set $D_k \times \{ i \}$ with $|i| \le k$ is clear.

From continuity of $f$ and the limsup condition, it
follows that for every $\ep > 0$  there exists $\d > 0$ such that
$d(x,e) < \d$ implies $d(f(x),e) < \ep$ and $x \in D_k$ with $\max(1/(|k| + 1), 1/(|\ell| + 1)) < \ep$.

For continuity at $(e,i)$ we note that if $d((x,j),(e,i)) < \min(\d, 1/(|i|+1))$,
then $j = i$ and so $|i| \le k$.  If $-k \le i < k$, then
$sf(x,i) = (f(x),i+1)$ is within $\ep$ of $(e,i+1)$. If $i = k$, then $sf(x,k) = (f(x),\ell)$ is within $\ep$ of $(e,k+1)$.

For continuity at $(e, *)$ we note that if $d((x,i),(e,*)) < \d$,
then $1/(|i|+1) < \d$ and so whether $i < k$ or $i = k$, $sf(x,i)$  is within $\ep$ of $(e,*)$.

So continuity of $sf$ at the points of $\{ e \}\times \Z^*$ follows.

If  $f^{-1}(x) \in D_{\ell'}$, then $(sf)^{-1}(x,-k) = (f^{-1}(x),\ell')$. Thus, $sf$ is bijective and so is
a homeomorphism.

The $T$ identity in (i) is easy to check. Again the countability condition in (ii)  is obvious.

The map $\pi$ does not map $sf$ to $f$, but maps the orbit of $(x,i)$ to a stuttering version of the orbit of $x$. That is, repeats of
$x$ until $i$ reaches $k$, then $2 \ell + 1$ repeats of $f(x)$ and so on.

To prove (iii) suppose $y \in D_{\ell'} \cap \om f(x)$. If $|j| \le \ell'$ and $\ep > 0$ is small enough that the $\ep$ ball centered at $y$ is contained
in $D_{\ell'}$, then there exists $p\in \N$ such that $d(f^p(x),y) < \ep$, and so $f^p(x) \in D_{\ell'}$. Thus, $(f^p(x),j)$ is a point of the $sf$
orbit of $(x,i)$ which is $\ep$ close to $(y,j)$.

If $e \in \om f(x)$ and $j \in \Z$, then if $\ep > 0$ is small enough, $d(y,e) < \ep$ will imply $y \in D_{\ell'}$
with $\ell' > |j|$. Hence, again if $d(f^p(x),e) < \ep$, then $(f^p(x),j)$ is in the $sf$ orbit of $(x,i)$. Hence, $(e,j) \in \om (sf)(x,i) $
for all $j \in \Z$. Since, $\om (sf)(x,i)$ is closed it contains $(e,*)$ as well.

Next assume $y \in D_{\ell'} \cap \NN f(x)$. For $\ep > 0$ small enough, the $\ep$ ball about $x$ is contained in $D_k$ and the
$\ep$ ball about $y$ is contained in $D_{\ell'}$. If $d(z,x) < \ep$ and $d(f^p(z),y) < \ep$, then $(f^p(z),j)$ is on
the $sf$ orbit of $(z,i)$ and is $\ep$ close to $(y,j)$. In the $y = e$ case we choose $\ep$ so that $d(u,e) < \ep$ implies $u \in D_{\ell'}$
with $\ell' > |j|$ and proceed as before. Again $\NN (sf)(x,i)$ is closed and so we pick up $(e,*)$.

Finally, if $y \in D_{\ell'} \cap \CC f(x)$, then we choose $\ep > 0$ small enough $\ep$ ball about $y$ is contained in $D_{\ell'}$.
 If $\{ x = x_0, \dots, x_p = y \}$ is an $\ep$ chain from $x$ to $y$, then we run along the
 orbit of $(x_0,i)$ continuing to $(f(x_0),-1)$ and then jumping to $(x_1,0)$.
Each jump occurs at the $0$ level until $f(x_{p-1})$ which is $\ep$ close to $y$ and so is in $D_{\ell'}$, then
with $x_{p-1} \in D_{\ell"}$ we jump from $(x_{p-1},\ell")$ to $(y,-\ell')$. Observe that $sf(x_{p-1},\ell") = (f(x_{p-1}),-\ell')$. We then run up to $(y,j)$. We leave the
argument for $e \in \CC f(x)$ to the reader as it follows the patterns above.

Items (iv) and (v) are obvious.

If $x \in Trans_f$, then by (iii), $(x,i) \in Trans_{sf}$. If $(X,f)$ is chain transitive, then by (iii) and (iv), for every $x \in X$,
$(e,*) \in \CC (sf)(x,i)$ and $(e,*) \in \CC (sf)^{-1}(x,i)$. So $(sX,sf)$ is chain transitive by (d) of Definition \ref{def1.01}.

Finally, (vii) is easy to check.

\end{proof}\vspace{.25cm}

The stretched suspension is a special case of the spin construction given in \cite{AG19} Section 1.5.\vspace{.25cm}

Our third construction is the pointed union.

Let $\{ (X_i,f_i) : i \in I \}$ be an indexed collection of pointed dynamical systems with $e_i$ the base point of $(X_i,f_i)$ and with
$I$ a countable index set. Let $Z = (\bigcup_i \ X_i)^*$ be the one point compactification of the disjoint union of the $X_i$'s. To assure disjointness
we can use $X_i \times \{ i \}$ but we will assume that we needn't bother. Using $f_i$ on each $X_i$ we obtain a homeomorphism $f_Z$ on $Z$. The
set $E = \{ * \} \cup \{ e_i : i \in I \}$ is a closed $f_Z$ invariant set. We then smash $E$ to a point $e$. That is,
we take the quotient space of $Z$ via the closed equivalence relation $(E \times E) \cup \{ (z,z) : z \in Z \}$. We denote the quotient space
by $\vee_i \ X_i$  and the homeomorphism induced from $f_Z$ by $\vee_i f_i$. We write $\vee_i (X_i,f_i)$ for this pointed union
$(\vee_i X_i, \vee_i f_i)$.
We identify $X_i$ with its copy in $\vee_i X_i$.

As an alternative description, we begin with the product dynamical system $(\prod_i  X_i, \prod_i  f_i)$ and let $\vee_i X_i$ be the
closed, invariant subset consisting
of those points $x$ such that $x_i \not= e_i$ for at most a single $i \in I$. The base point is the point $x = e$ with $x_i = e_i$ for all $i$.
%
\vspace{.25cm}

\begin{theo}\label{theo2.02} Let $\{ (X_i,f_i) : i \in I \}$ be a countably indexed collection of
pointed dynamical systems with $e_i$ the base point of $(X_i,f_i)$. The \emph{pointed union} $\vee_i (X_i,f_i)$
is a pointed dynamical system with base point $e$ which satisfies the following conditions.
 \begin{itemize}
 \item[(i)] The homeomorphism  $\vee_i f_i$  on $\vee_i X_i$ satisfies
 $(\vee_i f_i)^{-1} \ = \ \vee_i (f_i^{-1})$.

\item[(ii)] If each $X_i$ is countable, then $\vee_i X_i$ is countable.

 \item[(iii)] If $x \in X_j $, then $\om (\vee_i f_i)(x) = \om (f_j)(x) \subset X_j$ and if $x \not= e_j$, then
 $\NN (\vee_i f_i)(x) = \NN (f_j)(x) \subset X_j$. Furthermore, $\CC (\vee_i f_i)(x) \supset \CC (f_j)(x)$.
 In particular, if $x$ is recurrent, non-wandering or chain recurrent for $f_j$, then $x$ satisfies the
corresponding property for $\vee_i f_i$.

\item[(iv)] If $x \in X_j \setminus \{ e_j \}$, then $x$ is isolated in $\vee_i X_i$ if and only if it is isolated in $X_j$.

\item[(v)] If $(X_i,f_i)$ is chain transitive (or central) for all $i$, then $\vee_i (X_i,f_i)$ is chain transitive (resp. central).

\item[(vi)] The system $\vee_i (X_i,f_i)$ is trivial if and only if each $(X_i, f_i)$ is trivial.
\end{itemize}\end{theo}

\begin{proof} It is clear that $\vee_i X_i$ is compact and $\vee_i f_i$ is continuous. The space $\vee_i X_i$ is metrizable
because $I$ is countable.

The remaining results are easy to see. The one point to note is that  $X_j \setminus \{ e \}$
is an open set in $\vee_i X_i$ and so if $x \in X_j \setminus \{ e \}$ and $z \in \vee_i X_i$ is
close enough to $x$, then $z$ an element of $X_j \setminus \{ e \}$ as well  and thus the
orbit of $z$ remains in $X_j$. Hence, $\NN (\vee_i f_i)(x) \subset X_j$.

\end{proof}\vspace{.5cm}

\section{\textbf{The Cantor-Bendixson Depth and the Birkhoff Depths}} \vspace{.5cm}

Being a recurrent point is an intrinsic property.  If $x$ is a recurrent point for $f$ and $A$ is a closed invariant subset
of $X$ with $x \in A$, then $x$ is a recurrent point for the restriction $f|A$. The same need not be true of non-wandering points. Thus, in Proposition \ref{theo1.02} we showed that every point of $\om f(x)$ is non-wandering for $f$.  However, it need not be true that
the restriction $f|\om f(x)$ is central, i.e. that every point is non-wandering for $f|\om f(x)$. It is this issue which
leads to the phenomenon of Birkhoff depth.

For a  dynamical system we extend slightly the notation from the Introduction, defining:
\begin{align}\label{eq3.01}\begin{split}
z_{CB}(X) \ & = \ \{ x \in X : x \ \ \text{is not isolated in} \ \ X \}, \\
z_{BN}(X) \ & = \ \{ x \in X : x \ \ \text{is non-wandering in}  \ \ X \}, \\
z_{BL\pm}(X) \ & = \  \ol{ \bigcup_{x \in X} \ (\a f(x) \cup \om f(x))}, \\
z_{BL}(X) \  & = \ \ol{ \bigcup_{x \in X} \ \om f(x)},  \\
\xi(X) \ & =  \ \ol{Recur_f} \ = \ \ol{\{ x \in X : x \in \om f(x) \}}.
\end{split}\end{align}

Notice that if $x \in \a f(x)$, then $f^{-1}$ is topologically transitive on $\a f(x)$ and so $f$ is topologically transitive
on $\a f(x)$ as well. Hence, $\a f(x)$ contains a dense set of points recurrent for $f$. Thus,
$\xi(X)  =  \ol{\{ x \in X : x \in \a f(x) \}}$.  As it is the closure of the set  of recurrent points, $\xi(X)$ is the Birkhoff center for $f$
and the Brikhoff center for $f^{-1}$.

If $X$ does not contain any isolated periodic points, then
\begin{equation}\label{eq3.02}
\xi(X) \ \subset \ z_{BL}(X)\ \subset \ z_{BL\pm}(X)\ \subset \ z_{BN}(X)\ \subset \ z_{CB}(X).
\end{equation}

The inclusion $z_{BL\pm}(X) \subset  z_{BN}(X)$ is proper in the Shapovalov examples.
For example:
\begin{ex}\label{ex3.01}A subshift example. \end{ex}

With $s$ the shift homeomorphism on $2^{\Z} = \{ 0, 1 \}^{\Z}$, we denote by $[A]$ the closure of $\{ s^n(x) : x \in A, n \in \Z \}$ writing
$[x]$ for $[\{ x \}]$ when $x \in 2^{\Z}$. Let $e, a, b^k \in 2^{\Z}$ for $k \in \N$ be defined by $e_i = 0$ for all $i \in \Z$ and
$a_0 = 1, b^k_0 = b^k_k = 1$ and $a_i = b^k_i = 0$ otherwise. With $X = [\{ b^k : k \in \N \}]$ the set of non-wandering points is $[a]$ but the
only alpha or omega limit point is $e$.\vspace{.25cm}

It may also happen that
the inclusion $z_{BL}(X) \subset  z_{BL\pm}(X)$ is proper. \vspace{.25cm}

\begin{ex}\label{ex3.02} An alternative attachment. \end{ex}

Let $X = \Z^*$ with $f$ given by $f(i) = i+1$. Because $(X,f)$ is chain transitive we can choose $\{ y_k \}$ a full asymptotic chain
on $X$.  Now let $x_k = \begin{cases} y_k \ \ \text{for} \ \ k < 0, \\ * \ \ \text{for} \ \ k \ge 0. \end{cases}$. Define
$(aX, af)$ as in (\ref{eq1.11}) and (\ref{eq1.12}). For every $k \in \Z$, $\a f(x_k,k) = X \times \{* \}$ but $\om f(x_k,k) = (*,*)$.
Thus, $z_{BL\pm}(aX) = z_{BN}(aX) = X \times \{ * \},$ while $ z_{BL}(aX) = \{ (*,*) \}$.

For a subshift example, we define $e$ and $a$ as in Example \ref{ex3.01} and define $c \in 2^{\Z}$ by $c_i = 1 $ for  $i = -2^k, k \in \N$ and
$= 0$ otherwise. With $X = [c]$, we have $\a s(c) = [a]$, but $e$ is the only omega limit point.
 \vspace{.25cm}

Recall from Definition \ref{def0.01} that a system $(X,f)$ is simple when it is a countable,  pointed dynamical system
with base point $e$ the unique recurrent point and such that every non-isolated point is an omega limit point for some isolated point of $X$.
It is trivial when $X = \{ e \}$. From Lemma \ref{lem1.01a} we see that any simple system is chain transitive. For a simple system we have
\begin{equation}\label{eq3.03}
 z_{BL}(X)\ = \ z_{BL\pm}(X)\ = \ z_{BN}(X), \hspace{2cm}
\end{equation}
and we denote by $z(X)$ this common value. If, in addition, $X$ is non-trivial, then $z(X) = z_{CB}(X)$ as well.\vspace{.25cm}

\begin{df}\label{df3.03} A dynamical system $(X,f)$ is \emph{totally simple} when it is simple and for every ordinal $\a$
(with $X_0 = X$)
\begin{align}\label{eq3.04}\begin{split}
&X_{\a + 1} \ = \ z(X_{\a}),\\
\text{and the restriction} & \ \ (X_{\a + 1},f|X_{\a + 1}) \ \ \text{is simple}, \\
X_{\a} \ = \ & \bigcap_{\b < \a} \ X_{\b} \quad \text{for} \ \a \ \text{a limit ordinal},\\
\text{and the restriction} &  \ \ (X_{\a},f|X_{\a}) \ \ \text{is simple}.
\end{split}\end{align}\end{df}\vspace{.25cm}

When the system is totally simple and the sequence stabilizes at the ordinal $\th$, then
\begin{equation}\label{eq3.05}\begin{split}
X_{\th} = \xi(X) = \{ e \} \quad \text{with} \hspace{3cm} \\ \th_{BL}(X)\ = \ \th_{BL\pm}(X)\ = \ \th_{BN}(X) \ = \ \th, \\
\text{and} \quad \th_{CB}(X) \ = \ \th + 1. \hspace{2cm}
\end{split}\end{equation}

To build our examples, we use  our previous constructions.

\begin{theo}\label{theo3.04}(a) If $(X,f)$ is a simple system and $(aX, af)$ is an attachment construction as in Theorem \ref{theo1.05},
then $(aX, af)$ is simple with $z(aX) =  X \times \{ * \}$. If $(X,f)$ is totally simple, then $(aX, af)$ is totally simple
with $\th(aX) = 1 + \th(X)$.

(b) If $(X,f)$ is a simple system and $(sX, sf)$ is a stretched suspension construction as in Theorem \ref{theo2.01},
then $(sX, sf)$ is simple with $z(sX) = s(z(X))$. If $(X,f)$ is totally simple, then $(sX, sf)$ is totally simple
with $\th(sX) = \th(X) + 1$.

(c) If $\{ (X_i,f_i) : i \in I \}$ is a countably indexed list of simple systems and $\vee_i (X_i, f_i)$ is the pointed
union construction   as in Theorem \ref{theo2.02}, then $\vee_i (X_i, f_i)$ is simple with $z(\vee_i X_i) = \vee_i z(X_i)$.
If each $(X_i,f_i)$ is totally simple, then $\vee_i (X_i, f_i)$ is totally simple  with
$\th(\vee_i X_i) \ = \ \sup_i \th(X_i)$.
\end{theo}

\begin{proof} (a) From Theorem \ref{theo1.05} (ii) and (iv), it is clear that if $(X, f)$ is any countable, pointed dynamical
system with $e$ the unique recurrent point, then $(aX, af)$ is simple with $z(aX) =   X \times \{ * \}$. In particular, $(X,f)$ itself need not be simple.
After one step we are at a subsystem isomorphic to $(X,f)$. So if $(X,f)$ is totally simple, then $(aX, af)$ is totally simple with
$\th(aX) = 1 + \th(X)$.

(b) From Theorem \ref{theo2.01} (iii),  (iv) and (v) it is clear that if $(X, f)$ is simple, then $(sX, sf)$ is simple with
$z(sX) = s(z(X))$. Note that by (vii), the map $s(f|z(X))$ is the restriction of $sf$ to $s(z(X))$. If $(X,f)$ is totally simple,
then by induction $(sX)_{\a} = s(X_{\a})$ for all ordinals $\a < \th(X)$. Furthermore, at $\th = \th(X)$, $X_{\th} = \{ e \}$ and
so $(sX)_{\th} = \pi^{-1}(e)$ which we have denoted $s\{ e \}$. It is isomorphic to the translation on $\Z^*$ which is simple and which
reaches the trivial system $\{ (e, *) \}$ in the next step. Hence, $\th(sX) = \th(X) + 1$.

(c) Since a point of $X_j$ is recurrent in $\vee_i X_i$ if and only if it is recurrent in $X_j$, it is clear that, in general,
 $\xi(\vee_i X_i) = \vee_i \xi(X_i)$. From induction using Theorem \ref{theo2.02} (iii) and (iv) it is clear that for each of
 the three Birkhoff sequences and every ordinal $\a$, $(\vee_i X_i)_{\a} = \vee_i ((X_i)_{\a})$. From Theorem \ref{theo2.02} we see
  that $\vee_i (X_i, f_i)$  is simple if and only if all of the $(X_i, f_i)$'s are simple and that
$z(\vee_i X_i) = \vee_i z(X_i)$. Then $(\vee_i X_i)_{\a} = \vee_i ((X_i)_{\a})$ for all $\a$ implies that
$\vee_i (X_i, f_i)$  is totally simple if  all of the $(X_i, f_i)$'s are totally simple. From (vi) it follows that the
procedure terminates exactly when all the $(X_i, f_i)$ procedures have terminated and this occurs at $\sup_i \th(X_i)$.

\end{proof}\vspace{.25cm}

Now we construct our examples by using increasing transfinite sequences. For the first version we define
\begin{align}\label{eq3.06}\begin{split}
Y^0 \  = \ & \{ e \}, \\
(Y^{\a+1},g^{\a+1}) \ & = \ (aY^{\a},ag^{\a}),  \\
(Y^{\a}, g^{\a}) \  = \ \vee_{\b < \a} \  & (Y^{\b},  g^{\b}) \ \  \text{for} \ \   \a \ \ \text{a limit ordinal}.
 \end{split} \end{align}

 This does not do what we want because for an infinite ordinal $\a,$ \\ $ \ 1 + \a = \a$. The result is nonetheless interesting.
The transfinite sequence $(Y^{\a},g^{\a})$ for countable ordinals $\a \ge \om$ (the first infinite ordinal) consists of increasingly complicated
simple systems while for all such $\a, \ \ \th(Y^{\a}) \ = \ \om$.

For the examples we want, we use instead
\begin{align}\label{eq3.07}\begin{split}
X^0 \  = \ & \{ e \},  \\
(X^{\a+1},f^{\a+1}) \ & = \ (sX^{\a},sf^{\a}),  \\
(X^{\a}, f^{\a}) \  = \ \vee_{\b < \a} \  & ( X^{\b},  f^{\b}) \ \  \text{for} \ \   \a \ \ \text{a limit ordinal}.
 \end{split} \end{align}

 Now by induction using Theorem \ref{theo3.04}(b) and (c),we obtain for every countable ordinal $\th$
\begin{equation}\label{eq3.08}\begin{split}
\th_{BL}(X^{\th})\ = \ \th_{BL\pm}(X^{\th})\ = \ \th_{BN}(X^{\th}) \ = \ \th, \\
\text{and} \quad \th_{CB}(X^{\th}) \ = \ \th + 1. \hspace{2cm}
\end{split} \end{equation}
This proves Theorem \ref{theo0.02}.

It is possible to build uncountable examples $(X,f)$ such that for every countable ordinal $\th$
\begin{equation}\label{eq3.09}
\th_{BL}(X)\ = \ \th_{BL\pm}(X)\ = \ \th_{BN}(X) \ = \ \th_{CB}(X) \ = \ \th.
\end{equation}
For this one begins with $(X^0, f^0)$ a zero-dimensional, topologically transitive system, so that $X^0$ is perfect.
However, the procedure requires replacing the stretched suspension construction with the more general
spin construction from \cite{AG19}.

  \bibliographystyle{amsplain}

\begin{thebibliography}{10}


\bibitem{A93}
E. Akin,  {\bfseries The general topology of dynamical systems}, Grad. Studies in
Mathematics v. 01,(1993), Second printing (1996), Amer. Math. Soc., Providence. \vspace{.25cm}
\vspace{.25cm}

%

\bibitem{AC12}
 E. Akin, J. Carlson, \emph{Conceptions of topological transitivity}, Topology and Its Applications, (2012), \textbf{159}:2015--2030.	
\vspace{.25cm}

\bibitem{AG19}
E. Akin, E. Glasner, \textbf{WAP systems and labeled subshifts}, (2019),
Memoirs AMS 1265, Amer. Math. Soc., Providence. \vspace{.25cm}

\bibitem{ACS}
L. Alsed\`{a}, M. Chas, and J. Smital, \emph{On the structure of the $\om$-limit sets for continuous
maps of the interval}, International Journal of Bifurcation and Chaos, (1999), \textbf{09}:1719--1729.
\vspace{.25cm}

\bibitem{BK14}
E. Bilokopytov, S. F. Kolyada, \emph{Transitive maps on topological spaces},Ukrainian Mathematical Journal, (2014), \textbf{65}:1293--1318.
\vspace{.25cm}

\bibitem{B}
G. D. Birkhoff, \emph{Uber gewisse Zentralbewegungen Dynamischer Systeme}.
Ges. Wiss. G¨ottingen Nachr., Math.-Phys. Klasse (1926), 81--92.
\vspace{.25cm}

\bibitem{DF}
Y. N. Dowker, F. G. Friedlander  \emph{On limit sets in dynamical systems},
Proc. London Math. Soc. (1954),  \textbf{4}:168--176.
\vspace{.25cm}

\bibitem{K}
H. Kato, \emph{The depth of centres of maps on dendrites}, Journal of the Australian Mathematical
Society. Series A. Pure Mathematics and Statistics (1998), no. 1, \textbf{64}:44--53.
\vspace{.25cm}

\bibitem{M}
A. G. Maier, \emph{On the ordinal number of central trajectories}. Dokl. Akad.
Nauk SSSR 59 (1948),  no. 8, 1393--1396.
\vspace{.25cm}

\bibitem{N}
D. A. Neumann, \emph{Central sequences in dynamical systems}, Amer. J. Math.,  (1978), \textbf{100}:1--18.
\vspace{.25cm}

\bibitem{PF}
D. C. Perez Flores, \emph{Countable spaces and countable dynamics}, thesis, University of Birmingham, (2017).
\vspace{.25cm}

\bibitem{S65}
A. N. Sarkovskii, \emph{Attracting and attracted sets}, Dokl. Akad. Nauk SSSR (1965),
\textbf{160}:10361038. English translation in Sov. Math. Dokl. (1965), \textbf{6}:268--270.
\vspace{.25cm}

\bibitem{S00}
S. A. Shapovalov, \emph{A new solution of one Birkhoff problem}, J. of Dynamical and control systems,(2000),  \textbf{6}: 331--339.

\end{thebibliography}

\end{document}